\documentclass[10pt,a4paper,reqno]{amsart}

\usepackage{amsmath,amssymb,amsthm,amsfonts}
\usepackage{amsbsy}
\usepackage[utf8]{inputenc}
\usepackage[shortlabels]{enumitem}
\usepackage[normalem]{ulem}
\usepackage{cancel}
\usepackage{xspace}
\usepackage[greek,english]{babel}
\usepackage{url}
\usepackage{wrapfig}
\usepackage{enumitem}
\usepackage{multicol}
\usepackage{moreenum}
\usepackage{xcolor}
\usepackage{subcaption}
\usepackage{tikz}
\usepackage[bottom=2.5cm,top=2.5cm, left=2.5cm, right=2.5cm]{geometry}
\usepackage{array,multirow}
\usepackage[all,cmtip]{xy}
\definecolor{LinkColor}{rgb}{0,0,0} %black

\newcommand{\ii}{\mathbf{i}}
\newcommand{\jj}{\mathbf{j}}
\newcommand{\kk}{\mathbf{k}}

\usepackage[colorlinks=true,linkcolor=LinkColor,citecolor=LinkColor,urlcolor= LinkColor, naturalnames, hyperindex, pdfstartview=FitH, bookmarksnumbered, plainpages]{hyperref}

\usepackage{cleveref}

\makeatletter
\newcommand{\slunlhd}{%
  \mathrel{\mathpalette\sl@unlhd\relax}%
}

\newcommand{\sl@unlhd}[2]{%
  \sbox\z@{$#1\lhd$}%
  \sbox\tw@{$#1\leqslant$}%
  \dimen@=\ht\tw@
  \advance\dimen@-\ht\z@
  \ifx#1\displaystyle
    \advance\dimen@ .2pt
  \else
    \ifx#1\textstyle
      \advance\dimen@ .2pt
    \fi
  \fi
  \ooalign{\raisebox{\dimen@}{$\m@th#1\lhd$}\cr$\m@th#1\leqslant$\cr}%
}
\makeatother
 
\newtheorem{maintheorem}{Theorem}[]
\newtheorem{maincorollary}[maintheorem]{Corollary}

\newtheorem{theorem}{Theorem}[]

\newtheorem{lemma}[theorem]{Lemma}
\newtheorem{proposition}[theorem]{Proposition}
\theoremstyle{definition}

\newcommand{\cut}{\textsf{cut}\xspace}
\newcommand{\SL}{\operatorname{SL}}
\newcommand{\GL}{\operatorname{GL}}

\newcommand{\F}{\mathbb{F}}
\newcommand{\Z}{\textup{Z}}

\newcommand{\ZZ}{\mathbb{Z}}

\renewcommand{\O}{\mathcal{O}}
\newcommand{\Stab}{\operatorname{Stab}}
\newcommand{\Core}{\operatorname{Core}}

\definecolor{wildstrawberry}{rgb}{1.0, 0.26, 0.64}

\newcommand{\GEN}[1]{\left\langle #1 \right\rangle}

\newcommand{\qand}{\quad \text{and} \quad}

\title{The Gruenberg-Kegel graph of finite solvable rational groups}

\author{Sara C. Debón}
\author{Diego García-Lucas}
\author{\'{A}ngel del R\'{\i}o}

\begin{document}

 \begin{abstract}
    A finite group $G$ is said to be \textit{rational} if every character of $G$ is rational-valued. The \textit{Gruenberg-Kegel graph} of a finite group $G$ is the undirected graph whose vertices are the primes dividing the order of $G$ and an edge connects a pair of  different vertices $p$ and $q$ whenever $G$ has an element of order $pq$. In this paper, we complete the classification of the Gruenberg-Kegel graphs of finite solvable rational groups initiated in \cite{BKMdR}.
\end{abstract}

\maketitle

In this paper $G$ is a finite group.
The following  three  conditions are equivalent, and when they are satisfied we say that $G$ is \textit{rational}.
\begin{itemize}
 \item Every character of $G$ takes values in the field of rational numbers.
 \item For every element $g$ of $G$, all the generators of $\GEN{g}$ are conjugate in $G$.
 \item For every $g\in G$, $[N_G(\GEN{g}):C_G(g)]=\varphi(|g|)$, where $\varphi$ denotes the Euler totient function.
\end{itemize}

The \textit{Gruenberg-Kegel graph} of $G$, abbreviated \textit{GK-graph} of $G$, is the undirected graph $\Gamma_{GK}(G)$ whose vertices are the primes dividing the order of $G$ and  an edge connects a pair of  different vertices $p$ and $q$ whenever $G$ has an element of order $pq$.
Some authors use the term \textit{`prime graph'} to refer to the Gruenberg-Kegel graph, but since prime graphs appear in many different contexts, we will restrict the terminology to just GK-graph.

Several mathematicians contributed to the study of the Gruenberg-Kegel graph of some classes of finite groups. For instance, if $G$ is rational and solvable, then the set of possible vertices of $\Gamma_{GK}(G)$ is contained in $\{2,3,5 \}$ (see \cite{Gow}); while if $\Gamma_{GK}(G)$ is a tree, then its size is bounded by $8$ (see \cite{LucidoTree}).
In \cite{GKLNS}, the graphs that occur as the Gruenberg-Kegel graph of a finite solvable group are characterized in a purely graph-theoretical way.
An analogue for finite groups whose order is divisible by at most five different primes is obtained in \cite{GKKM14}.
Theorem~D of \cite{BKMdR} contains an attempt to classify the GK-graphs of finite solvable rational groups and only for one graph it remained to be decided whether it is the GK-graph of a finite solvable rational group or not.  The aim of this paper is to complete this attempt. Technically, we answer negatively \cite[Question F]{BKMdR} and hence finish the proof of the following theorem:

\begin{maintheorem}\label{maintheorem:rational}
%The list of Gruenberg-Kegel graphs of non-trivial finite solvable rational groups is the following:
The Gruenberg-Kegel graphs of non-trivial finite solvable rational groups are precisely the following:

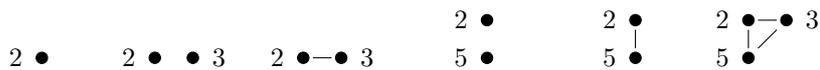
\begin{figure}[ht!]
\begin{tabular}{ccccccc}
		\begin{subfigure}{.1\textwidth}
			\centering 
			\begin{tikzpicture};
			\node[label=west:{$2$}] at (0,0) (2){};
			\foreach \p in {2}{
				\draw[fill=black] (\p) circle (0.075cm);
			}
			\end{tikzpicture}
		\end{subfigure}
&		\begin{subfigure}{.1\textwidth}
			\centering 
			\begin{tikzpicture}
			\node[label=west:{$2$}] at (0,1) (2){};
			\node[label=east:{$3$}] at (0.5,1) (3){};
			\foreach \p in {2,3}{
				\draw[fill=black] (\p) circle (0.075cm);
			}
			\end{tikzpicture}
		\end{subfigure}
&		\begin{subfigure}{.1\textwidth}
			\centering 
			\begin{tikzpicture}
			\node[label=west:{$2$}] at (0,1) (2){};
			\node[label=east:{$3$}] at (0.5,1) (3){};
			\foreach \p in {2,3}{
				\draw[fill=black] (\p) circle (0.075cm);
			}
			\draw (2)--(3);
			\end{tikzpicture}
		\end{subfigure} 
&		\begin{subfigure}{.1\textwidth}
			\centering
			\begin{tikzpicture}
			\node[label=west:{$2$}] at (0,0.5) (2){};
			\node[label=west:{$5$}] at (0,0) (3){};
			\foreach \p in {2,3}{
				\draw[fill=black] (\p) circle (0.075cm);
			}
			\end{tikzpicture}
		\end{subfigure}
& \begin{subfigure}{.1\textwidth}
			\centering 
			\begin{tikzpicture}
			\node[label=west:{$2$}] at (0,0.5) (2){};
			\node[label=west:{$5$}] at (0,0) (5){};
			\foreach \p in {2,3}{
				\draw[fill=black] (\p) circle (0.075cm);
			}
			\draw (2)--(3);
			\end{tikzpicture}
		\end{subfigure}
& \begin{subfigure}{.1\textwidth}
			\centering
			\begin{tikzpicture}
			\node[label=west:{$2$}] at (0,0.5) (2){};
			\node[label=east:{$3$}] at (0.5,0.5) (3){};
			\node[label=west:{$5$}] at (0,0) (5){};
			\foreach \p in {2,3,5}{
				\draw[fill=black] (\p) circle (0.075cm);
			}
			\draw (2)--(3);
			\draw (2)--(5);
			\draw (3)--(5);
			\end{tikzpicture}
		\end{subfigure}	
\end{tabular}	
\end{figure}
\end{maintheorem}

As a consequence we have the following:

\begin{maincorollary}
If $G$ is a finite solvable rational group of order divisible by $15$, then $G$ has elements of order $6$, $10$ and $15$.
\end{maincorollary}

We start fixing some basic notation and recalling some definitions. 
We use standard group theoretical notation. For example, $N_G(H)$ denotes the normalizer of a subgroup $H$ in $G$ and $C_G(X)$ denotes the centralizer of a subset $X$ in $G$.
Moreover, $C_n$  and  $ Q_n$ denote the cyclic and quaternion groups of order $n$, and $ \GL(n,q)$ and $\SL(n,q)$ denote the general and special linear $n$-dimensional groups over the field $\F_q$ with $q$ elements.
If $G$ is a group and $p$ a prime, then $G_p$ denotes a Sylow $p$-subgroup of $G$.

By default a $G$-module is a right $G$-module.
We use exponential notation for the action of $G$ on a $G$-module $V$ and multiplicative notation for the internal operation in $V$, i.e. we consider $V$ as an abelian multiplicative group, and if $g\in G$ and $v\in V$, then $v\cdot g=v^g$. Then $V\rtimes G$ denotes the corresponding semidirect product, so that $v^g=g^{-1}vg$.
In this paper we use always this notation for an elementary abelian $5$-group $V$.
Then we consider $V$ as a right $\F_5G$-module, where $\F_5=\ZZ/5\ZZ$, the field with $5$ elements.
In other words, if $v\in V$, $\alpha\in \F_5$ and $g\in G$, then $v\cdot \alpha =v^\alpha$, where $\alpha$ is identified with an integer modulo $5$, and $v\cdot g=v^g=g^{-1}vg$.
Although we consider $V$ as a multiplicative group inside the semidirect product $V\rtimes G$, we take advantage of considering $V$ as a vector space over $\F_5$ to use standard linear algebra notation. For example, in some cases $V$ is going to be identified with $\F_5^n$ for some integer $n$ and if $K$ is the kernel of the action of $G$ on $V$, then $G/K$ is identified with a subgroup of $\GL(n,5)$ (see e.g., \Cref{Heg4.9} below).

We say that $V_G$ has the \emph{eigenvector property} if for every $\alpha\in \F_5\setminus \{0\}$, every element $v$ of $V$ is an eigenvector with eigenvalue $\alpha$ for some $g\in G$, i.e. $v^g=v^\alpha$.
As the group of units of $\F_5$ is generated by $2$, $V_G$ has the eigenvector property if and only if for every $v\in V$, there is some $g\in G$ such that $v^g=v^2$.

If $H$ is a subgroup of $G$ and $W$ is an $H$-module, then $W^G$ denotes the induced $G$-module, i.e. $W^G= \bigoplus_{t\in T} W^t$, where $T$ is a right transversal of $H$ in $G$ and if $w\in W$, $t\in T$, $g\in G$ and $tg=hr$ with $h\in H$ and $r\in T$, then $(w^t)^g=(w^h)^r$.
Observe that if $w\in W\setminus \{1\}$ and $g\in G$, then $w^g\in W$ if and only if $g\in H$. In particular, $\Stab_G(W)=H$.

By a theorem by P. Hegedüs, a Sylow $5$-subgroup of a finite solvable rational group is normal and elementary abelian \cite{Heg05}. Our proof relies on the following proposition which is one of the main ingredients in the proof of Hegedüs' Theorem.

\begin{proposition}[\cite{Heg05}, Section 4.9]\label{Heg4.9}
    Let $G$ be a finite solvable rational group such that the Sylow $5$-subgroup $V$ of $G$ is minimal normal. Let $S$ be a complement of $V$ in $G$, and regard $V$ as an $\F_5 S$-module. Let  $H$  be a minimal subgroup of $S$ such that there exists an $\F_5  H$-submodule $W$ of $V$ with $V=W^{S}$ and the Brauer character of $W$ is rational. Let $K=C_S(W)$. Then one of the following holds:
    \begin{enumerate}
        \item \label{casea} $H/K\cong Q_8$ and $W=\F_5^2$.
        \item \label{caseb} $H/K\cong C_3\rtimes C_4$ and $W=\F_5^2$.
        \item\label{casec} $H/K\cong \SL(2,3)$ and $W=\F_5^2$.
        \item\label{cased} $H/K$ is isomorphic to the subgroup   of $\GL(4,5)$ generated by
        \begin{equation}\label{eq:matrices} \alpha=\left(\begin{matrix}
                                3&3&0&0 \\
                                4&1&0&0 \\
                                0&0&1&1 \\
                                0&0&2&3 
                                \end{matrix}\right),  \quad \beta=\left(\begin{matrix}0&0&2&3 \\
                                0&0&3&3 \\
                                1&4&0&0 \\
                                4&4&0&0\end{matrix}\right), \qand 
           \gamma=\left(\begin{matrix}
                                1&0&0&0 \\
                                0&3&0&0 \\
                                0&0&1&0 \\
                                0&0&0&2
                                \end{matrix}\right) ,
        \end{equation}
        where the action by conjugation on $W=\F_5^4$ is given by right  multiplication. 

        \item \label{casee}$H/K$ is isomorphic to the subgroup of $\GL(4,5)$ generated by $\alpha,\beta$ and $\gamma^2$, where the action by conjugation on $W=\F_5^4$ is given by right multiplication.
        \item \label{casef} $H/K$ has order $144$.
    \end{enumerate}
\end{proposition}

%\begin{theorem}[\cite{BKMdR}, Theorem D]\label{BKMdRD}  Let $G$ be a non-trivial finite solvable rational group. Then $\Gamma_{GK}(G)$ is one of the graphs (a), (c), (d), (e), (f ), (i) or (k) in Figure 1. Moreover, any of these graphs, except possibly (i), is realizable as the $GK$-graph of a solvable rational group.
%\end{theorem}

 A group $G$ is said to be \cut if for every $g\in G$, every generator of $\GEN{g}$ is conjugate to $g$ or $g^{-1}$ in $G$. Another important tool for our proof is the following lemma.

\begin{lemma}[\cite{BKMdR}, Lemma 3.5]\label{BKMdR3.5} Let $G$ be a finite \cut group and let $G_p$ and $G_q$ be Sylow subgroups of $G$ for two distinct primes $p$ and $q$ dividing the order of $G$. Suppose that $G_p$ is normal in $G$ and $G$ does not contain an element of order $pq$. Then $G_q$ is either the quaternion group of order 8 or a cyclic group of order dividing 4 or $q$.
\end{lemma}

We start now the proof of \Cref{maintheorem:rational}. By Theorem~D in \cite{BKMdR}, it suffices to prove that $(3-2-5)$ is not the GK-graph of a finite solvable rational group. By means of contradiction we fix a finite solvable rational group $G$ of minimal order with $\Gamma_{GK}(G)=(3-2-5)$. By Hegedüs Theorem, $G$ has a unique Sylow $5$-subgroup $V$, which is   elementary abelian. Then $G=V\rtimes S$, where $S$ is a $\{2,3\}$-Hall subgroup $S$ of $G$. Thus we consider $V$ as a right $\F_5S$-module as explained above. Moreover, $S$ is rational since the class of solvable rational groups is closed under epimorphic images.

The rationality condition implies that $V_S$ has the eigenvector property. In other words, for every $v\in V$ there exists some $s\in S$ so that $v^s=v^2$.

Moreover, if $U$ is a normal subgroup of $G$ such that $U$ is properly contained in $V$, then $G/U$ is a rational solvable group whose GK-graph is $(3-2-5)$ by Theorem~D in \cite{BKMdR}. Then $U=1$ because of the minimality of $G$. Thus, $V$ is a  minimal normal subgroup of $G$, or equivalently, $V_{\F_5S}$ is simple. 

As every rational group is \cut, \Cref{BKMdR3.5} applies to $G$ and hence a Sylow $3$-subgroup of $S$ is cyclic of order $3$.
We fix a Sylow $2$-subgroup $S_2$ of $S$.
Observe that $S_2$ is not normal in $S$, since otherwise $S/S_2\cong C_3$ would be rational.
Thus $S$ has exactly three Sylow $2$-subgroups, namely $S_2$, $S_2^a$ and $S_2^{a^2}$, where $a$ is any element of $S$ of order $3$.
As $G$ has not elements of order $15$, if $v\in  V\setminus\{1\}$, then $v^a\neq v$.

The action of $S$ on $V$ is faithful. Indeed, otherwise the kernel $N$ of this action is a normal non-trivial $2$-subgroup of $G$ such that $G/N$ is a solvable rational group with $\Gamma_{GK}(G/N)=(3-2-5)$ by Theorem~D in \cite{BKMdR}, in contradiction with the minimality hypothesis.

\begin{lemma}\label{3elH} Let $H$ be a subgroup of $S$ and $W$ a $\mathbb{F}_5H$-submodule of $V$. If $V=W^S$, then the action of $H$ on $W$ has the eigenvector property and $H$ contains all the elements of order $3$ of $S$.
\end{lemma}
\begin{proof}
Let $a$ be an element of $S$ of order $3$ and $w\in W\setminus \{1\}$. Take $v=ww^aw^{a^2}$. As $v^a=v$, necessarily $v=1$ and hence $a\in H$. 
Moreover, from the eigenvector property of $V_S$, there exists $s\in S$ such that $w^s= w^2\in W\cap W^s$, and necessarily $s\in H$, as desired.
\end{proof}

From now on, $H$ is a minimal subgroup of $S$ such that there exists an $\mathbb{F}_5H$-submodule $W$ with $V=W^S$ and with rational Brauer character. Then the kernel $K$ of the action of $H$ on $W$ is a $2$-group. Applying Lemma \ref{3elH} we conclude that $H$ contains all the elements of order $3$ of $S$ and has the eigenvector property on $W$. 
We fix an element $a$ of $S$ of order $3$.
Let $A$ denote the subgroup of $S$ generated by the elements of order $3$;
then $A$ is the smallest normal subgroup of $S$ containing $a$, and $|A|=3$ if and only if $S$ has a normal Sylow $3$-subgroup.
Otherwise $A/A'$ is an abelian $3$-group with $[A:A']=3$.
Therefore, in any case $A_2=A'$ is the unique Sylow $2$-subgroup of $A$, and hence it is normal in $S$.
Moreover, as $\GEN{a}$ is a Sylow 3-subgroup of $A$ and as $A$ is normal in $S$,   the Frattini argument implies that $S=AN_S(\GEN{a})$.
As $A\subseteq H$, we also have $S=HN_ S(\GEN{a})$ and hence $N_S(\GEN{a})$ contains a right transversal $T$ of $H$ in $S$ containing $1$ which will be fixed throughout. So $V=\bigoplus_{t\in T} W^t$.

Moreover, $H/K$ is one of the groups of Proposition \ref{Heg4.9}. Observe that cases \eqref{casea} and \eqref{casef} can be immediately excluded, as the former does not contain elements of order $3$, while in the latter $9$ divides $|H/K|$ but a Sylow $3$-subgroup of $S$ has order $3$, by Lemma \ref{BKMdR3.5}. 
Thus from now on $W$, $H$ and $K$ are as in one of the cases \eqref{caseb}, \eqref{casec}, \eqref{cased} and \eqref{casee} of \Cref{Heg4.9}. In \Cref{fact:removeCasea} we prove that \eqref{caseb} cannot hold, and in \Cref{Casosde} we prove that cases \eqref{cased} and \eqref{casee} are also not possible. The proof finishes by dismissing case \eqref{casec}, which requires an ad hoc argument.

\begin{lemma} \label{Kcorefree}
$\Core_S(K)=\bigcap_{s\in S_2}K^s=\bigcap_{t\in T}K^t=1$.
\end{lemma}

\begin{proof} 
First of all, by Lemma \ref{3elH}, $\GEN{a}\subseteq H\subseteq N_S(K)$.
Moreover, $\GEN{a}$ is a left transversal of $S_2$ in $S$. Thus $\Core_S(K)=\bigcap_{s\in S_2} \bigcap_{i=0}^2 K^{a^is}=\bigcap_{s\in S_2} K^s$.

Now it suffices to see that $\bigcap_{t\in T} K^t=1$. Take $x\in \bigcap_{t\in T} K^t$ and for every $t\in T$ let $k_t=txt^{-1}$.
Let $v$ be an arbitrary element of $V$. Then $v=\prod_{t\in T}{w_t}^{t}$ for unique $w_t\in W$ for each $t\in T$. As each $k_t\in K$,
		$$v^x=\prod_{t\in T}{w_t}^{tx}=\prod_{t\in T}{w_t}^{k_tt}=\prod_{t\in T}{w_t}^{t}=v.$$
Therefore $x$ belongs to the kernel of the action of $S$ on $V$. As this action is faithful, $x=1$.
\end{proof} 

\begin{lemma}\label{A=3} 
$|A|=3$ if and only if $H/K\cong C_3\rtimes C_4$.
\end{lemma}

\begin{proof}
Clearly $|A|=3$ if and only if $A=\GEN{a}$ if and only if $\GEN{a}$ is normal in $S$. In that case, $K\GEN{a}/K$ is a normal Sylow $3$-subgroup of $H/K$ and considering the four options (b), (c), (d) and (e) of
\Cref{Heg4.9} it follows that $H/K\cong C_3\rtimes C_4$.
Conversely, suppose that $H/K\cong C_3\rtimes C_4$.
Then $A\subseteq \GEN{K,a}\lhd H$ and hence $[A,\GEN{a}]\subseteq [K,\GEN{a}] \subseteq K$.  
Moreover, as $T\subseteq N_S(\GEN{a})$, for every $t\in T$ we have $[A,\GEN{a}]=[A^t,\GEN{a^t}]\subseteq K^t$, and therefore $[A,\GEN{a}]\subseteq \bigcap_{t\in T} K^t=\Core_S(K)=1$, by \Cref{Kcorefree}. Therefore, $a$ is central in $A$ and as $A$ is generated by elements of order $3$ and its Sylow 3-subgroup has order $3$ necessarily $A=\GEN{a}$. 
\end{proof}

We define $L$ as the smallest subgroup of $H$ containing $K$ and the elements of order $3$ of $S$, that is $L=KA$.
Note that $L$ and $L_2=KA_2$ are both normal in $H$.
By inspecting the four options for $H/K$ it follows that
\begin{equation}\label{LK}
L/K\cong \begin{cases} C_3, & \text{if } H/K\cong C_3\rtimes C_4; \\
	\SL(2,3), & \text{otherwise}.\end{cases}
\end{equation}

\begin{lemma}\label{Z_not_trivial}
  $Z(S)\neq 1$ and if $H/K\not\cong C_3\rtimes C_4$, then  $Z(S)\cap A_2\cong C_2$.
\end{lemma}

\begin{proof}
Suppose that $H/K\cong C_3\rtimes C_4$. Then the Sylow $3$-subgroup $\langle a \rangle$ of $S$ is normal in $S$, by \Cref{A=3}, and therefore so is $C_S(a)_2$ by the Frattini argument. Moreover, $|S|\ge 12$ and by the rationality hypothesis, $[S:C_S(a)]=[N_S(\GEN{a}):C_S(a)]=2$. Thus $C_S(a)_2\ne 1$.
In particular, $1 \neq C_S(a)_2 \cap Z(S_2)\subseteq Z(S)$, as desired.

Suppose that $H/K\not\cong C_3\rtimes C_4$.  By \Cref{A=3}, $A_2$ is a non-trivial Sylow $2$-subgroup of $A$ and hence $L_2=KA_2$ is the unique Sylow $2$-subgroup of $L$.
As $A_2$ is normal in $S$, there is $z\in A_2\cap Z(S_2)$ with $z\neq 1$.

We claim that $z\not \in K^s$ for every $s\in S_2$. Indeed, if $s\in S_2$ with $z\in K^s$, then for every $x\in S_2$ we have $z=z^{s^{-1}x}\in K^x$. Thus $z\in \bigcap_{x\in S_2}K^x=1$, by Lemma \ref{Kcorefree}, yielding a contradiction.

Let $s\in S_2$ and recall that $L/K\cong \SL(2,3)$ and hence $|Z(L/K)|=|Z(L_2/K)|=2$. 
Since $z\in A_2\cap Z(S_2)$, we have $z\in Z(L_2^s)$.
Then $K^sz\in Z(L_2^s/K^s)$,  $a\in A=A^s\subseteq L^s$ and $| K^sa|=3$. Thus $|K^sz|=2$ and $|K^sza|=6$ because $z\not \in K^s$. Denote by $z_2$ the $2$-part of $za$ and $z_3$ its $3$-part. Thus, $K^sza=K^sz_2z_3$ and $K^sz=K^sz_2$ because of uniqueness of the $2$-part.
This shows that $zz_2^{-1}\in K^s$ for every $s\in S_2$.
By \Cref{Kcorefree}, $z=z_2$ and hence $z_3=z_2^{-1}za=a$, so that $z$ commutes with $a$. It follows that $   z\in Z(S_2)\cap C_S(a) \subseteq Z(S)$.  Thus $1\neq \GEN{z}\subseteq Z(S)\cap  A_2$. The result follows since  $|Z(S)\cap A_2|\leq 2$ because  $\Z(S)\cap A_2\subseteq L$ and, as $Z(S)\cap K=1$, by \Cref{Kcorefree}, we have that  $Z(S)\cap A_2\cong  K(Z(S)\cap A_2) /K \subseteq Z(L/K)\cong C_2$.
\end{proof}

The following lemma also holds in the case where $H/K\cong \SL(2,3)$ but since its proof uses a different argument, we delay it until \Cref{lema:existed(ificil)}.

\begin{lemma} \label{lem:existd}
Let $H/K$ be one of the groups in cases \eqref{caseb}, \eqref{cased} or \eqref{casee} of \Cref{Heg4.9}. Then $C_S(a) \not \subseteq H$.
\end{lemma}
		
\begin{proof}
Observe that $L_2a$ is not central in $H/L_2$, $S=A_2N_S(\GEN{a})$ and $H$ is not rational.
Thus $H\cap N_S(\GEN{a})\not\subseteq C_S(a)$ and $N_S(\GEN{a})\not\subseteq H$.
Since $[N_S(\GEN{a}):C_S(a)]=2$, the lemma follows.
\end{proof}
%
%   	Let $H/K$ be one of the groups \eqref{caseb}, \eqref{cased} or \eqref{casee} of \Cref{Heg4.9}. By \Cref{lem:existd} and since every element of order $3$ is in $H$, we may assume that the transversal $T$ of $H$ in $S$ contains a $2$-element $d$ in $C_S(a)\setminus H$, which will be fixed throughout for these cases.

It is clear that $C_S(a)$ has a unique Sylow $2$-subgroup which we denote $C_S(a)_2$.

	\begin{lemma}\label{fact:removeCasea} $H/K$ is not isomorphic to $C_3\rtimes C_4$.
	\end{lemma}
	\begin{proof} 
		Suppose that $H/K\cong C_3\rtimes C_4$. 
		Let $u\in W\setminus\{1\}$. By \Cref{3elH}, $H$ contains an element $b$ such that $u^b=u^2$. Then $| Kb|=4$, $H=\GEN{K,a,b}$ and $a^b\in Ka^2$. 
		Define $v=u^a$ and observe that $\{u,v\}$ is a basis of $W$ since conjugation by $a$ has no fixed points in $W\setminus\{1\}$. As $G$ has not elements of order $15$, $au$ has order $3$ and hence $1=(au)^3 = u^{a^2}u^au=v^avu$. Thus $v^a=u^{4}v^{4}$. Moreover, $v^b=u^{ab}=u^{ba^2}=(u^2)^{a^2}=(v^a)^2=u^{3}v^3$.  Thus the $H$-orbit of $u$ in $W$ is $\mathcal{O}_u=\{ u^i, v^i, (uv)^i : i=1,\dots, 4  \}$.
		
		Furthermore, $H/K$ has exactly $6$ elements of order $4$, while $W$ has exactly $6$ subspaces of dimension $1$.
		By the eigenvector property of $W_H$, for every $w\in W\setminus \{1\}$,
		the set $X_{w}=\{Kh\in H/K : w^h=w^2\}$ is not empty.
		In addition, no element of order $4$ of $H/K$ is central and hence $X_{u}\cap X_{v}=\emptyset$ whenever $W=\GEN{u,v}$.
		Therefore, $|X_w|=1$ for every $w\in W\setminus \{1\}$ and hence there is a bijection from the set of one-dimensional subspaces of $W$ to the set of elements of order $4$ of $H/K$, associating $\GEN{w}$ with the unique element of $X_w$.
		
By \Cref{lem:existd}, we may assume that $T\setminus \{1\}$ contains an element $d$ in $C_S(a)_2$.
		By the eigenvector property of $V_S$, there is a $2$-element $s\in S$ such that $(u (u^2v)^d)^s=(u (u^2v)^d)^2$. If $u^s=((u^2v)^d)^2$ then $u^{sd^{-1}}=u^4v^2\in W $, so $sd^{-1}\in H$ and it follows that $u^4 v^2\in \mathcal{O}_u$, a contradiction.
		Hence $u^s=u^2$ and $(u^2v)^{ds}= ((u^2v)^d)^2= (u^4v^2)^d$. 
		Thus $s,dsd^{-1}\in H$, $X_u=\{Kb\}=\{Ks\}$ and $X_{u^2v}=\{Kdsd^{-1}\}$.

		Notice that $C_S(a)$, and therefore $C_S(a)_2$, are both normal in $S$ by \Cref{A=3}. As $d\in C_S(a)_2$ also $[s,d^{-1}]\in C_S(a)_2\cap H$. Thus $K[s,d^{-1}] \in KC_H(a)_2/K\subseteq C_{H/K}(Ka)_2=\GEN{Kb^2}$.
		Therefore $w^{[s,d^{-1}]}\in \{w,w^{-1}\}$ for every $w\in W$. 
		Moreover, $s=kb$ for some $k\in K$ and $ u^4v^2=(u^2v)^{dsd^{-1}}  =(u^2v)^{s [s,d^{-1}]} =(u^2v)^{kb[s,d^{-1}]} =(u^2v^3)^{[s,d^{-1}]}\in \{ u^2v^3, u^3 v^2 \}$, the desired contradiction. 
	\end{proof} 
	
	In the latter proof we used that $C_S(a)_2$ was normal to dismiss case \eqref{caseb}. In the proof for cases \eqref{cased} and \eqref{casee} we will use the same idea with the normal subgroup $A_2C_S(a)_2$ instead of $C_S(a)_2$.  
	
	\begin{lemma}\label{Casosde}
		$H/K$ is neither isomorphic to  $\GEN{\alpha,\beta,\gamma}$ nor $\GEN {\alpha,\beta,\gamma^2}$.
	\end{lemma}
	\begin{proof}
	Suppose that $H/K$ is isomorphic to one of the groups in the statement, acting on $W=\mathbb{F}_5^4$ as in \Cref{Heg4.9}.
	To simplify notation we identify $H/K$ with the corresponding group of matrices. We start with a series of basic facts about these groups and their actions that follow from straightforward computations on {\sf GAP} \cite{GAP4}. The subgroup of $H/K$ generated by the $3$-elements of $H/K$ and its Sylow $2$-subgroup are
		$$L/K= \GEN{\alpha , \gamma ^{-1}\alpha \gamma} \qand (L/K)_2=\GEN{\alpha \gamma^{-1} \alpha \gamma, \gamma^{-1} \alpha \gamma \alpha},$$ 
	respectively. Moreover
        $$C_{H/K}(\alpha)_2=\begin{cases} \GEN{\alpha \gamma \beta \alpha, \alpha \beta \gamma (\alpha \beta)^2},& \text{if }  H/K \text{ is } \GEN{\alpha, \beta, \gamma}; \\ \GEN{\alpha \gamma \beta \alpha^2 \beta \gamma (\alpha \beta)^2},&  \text{if }  H/K \text{ is } \GEN{\alpha, \beta, \gamma^2}, \end{cases}$$
 and $C_{H/K}(\alpha)_2$ is normal in $H/K$.
		The elements 
		$$u=(0,1,1,1)\qand v=(0,1,1,2)$$ 
		of $W$ lie in two different $H/K$-orbits. Moreover, the matrices $$ \mu_u=\left(\begin{matrix}
			0&0&3&2 \\
			0&0&2&2 \\
			4&1&0&0 \\
			1&1&0&0
		\end{matrix} \right) \qand \mu_v= \left(\begin{matrix}
			0&0&4&2 \\
			0&0&2&4 \\
			3&1&0&0 \\
			1&3&0&0
		\end{matrix}\right)$$
		belong to $\GEN{\alpha ,\beta ,\gamma^2}$ and they are the unique matrices in $\GEN{\alpha ,\beta,\gamma}$ such that $u\cdot \mu _u=2\cdot u$ and $v\cdot \mu_v=2\cdot v$.
		Here we are using additive notation.
		Moreover 
		\begin{equation}\label{MuMv}
			\mu_u^{-1}\mu_v=\left(\begin{matrix} 
				4&1&0&0 \\
				2&2&0&0  \\
				0&0&2&3 \\
				0&0&4&4 
			\end{matrix}\right) \notin (L/K)_2C_{H/K}(\alpha)_2.
		\end{equation}
		
Now we put together all this information to find a contradiction. Since $C_{H/K}(\alpha)_2$ is normal in $H/K$ and $\GEN{Ka}$ and $\GEN{\alpha}$ are conjugate in $H/K$, we derive that $C_{H/K}(\alpha)_2=C_{H/K}(Ka)_2$.

By \Cref{lem:existd}, we may assume that $T\setminus \{1\}$ contains an element $d$ in $C_S(a)_2$.
         In particular, $d\in A_2C_S(a)_2$.
		By rationality, there is an element $s\in S$ such that $(uv^d)^s=(uv^d)^2$. 
		If $u^s=(v^d)^2$ and $v^{ds}= u^2$ then $sd^{-1}\in H$, so $u$ and $v^2$ lie in the same $H/K$-orbit and hence so do $u$ and $v$, yielding a contradiction. 
		Then $u^s=u^2$ and $v^{ds}=(v^d)^2$. Thus $v^{dsd^{-1}}=v^2$ and $s, dsd^{-1} \in H$. Therefore by uniqueness,
		$$\mu_u=Ks \qand \mu_v=Kdsd^{-1}= Ks[s,d^{-1}]= \mu_u \cdot K[s, d^{-1}].$$
		As $A_2C_S(a)_2$ is normal in $S$ and $d\in A_2C_S(a)_2$, we have $[s,d^{-1}]\in A_2C_S(a)_2\cap H=A_2(C_S(a)_2\cap H)=A_2C_H(a)_2$ and hence $\mu_u^{-1}\mu_v=K[s,d^{-1}]\in (L/K)_2C_{H/K}(Ka)_2$, in contradiction with \eqref{MuMv}.
	\end{proof}
	
\underline{Thus necessarily $H/K\cong \SL(2,3)$ and $W=\mathbb{F}_5^2$.}
This case appears to be different from the previous ones, as in those our strategy consisted in finding elements $u,v\in W\setminus\{1\}$ lying in different $H$-orbits, and then exploiting this fact and that the element $uv^d$ is rational to find a contradiction, while, in the present setting, $H$ acts transitively on $W\setminus \{1\}$. This situation requires a new strategy, this time exploiting that the action of $H_2$ on $W\setminus \{1\}$ has three different orbits, which we use to construct an element $\mathbf{v}\in V$ in a  slightly less straightforward manner.
Then, by rationality there is a $2$-element $s\in S$ such that $\mathbf{v}^s=\mathbf{v}^2$.
We conclude the proof using $s$ to construct a non-trivial element in $C_S(a)\cap \left( \bigcap_{t\in T_0} K^t\right)$, contradicting \Cref{notClaimOmega} below.

    We first recall a series of facts: $H=L=KA$, and, by \Cref{Z_not_trivial}, $Z(S)\cap A_2\cong C_2$. We define $z$ to be the non-trivial central element of $S$ contained in $A_2$.
By \Cref{Kcorefree}, $z\not\in K$  and hence $Kz$ is the unique element of order $2$ of $H/K$.
    
As $S$ is rational, there exists a $2$-element $b\in N_S(\GEN{a})$ such that $a^b=a^2$. Moreover, $b\not \in H$, since $C_{H/K}(Ka)=N_{H/K}(Ka)$.
Then $S=AC_S(a)\GEN{b}$.
We define $N:=AC_S(a)$, which has index $2$ in $S$.
Clearly, $N$ and its unique Sylow $2$-subgroup $N_2$ are normal in $S$, and we can assume that $S_2=N_2\GEN{b}$.
Notice that a transversal $T_0$ of $H_2$ in $N_2$ is also a transversal of $H$ in $N$, so we can assume that $T=T_0 \cup T_0b$, where $T_0$ contains $1$ and is contained in $C_S(a)_2$. Let $V_0:=\bigoplus_{t\in T_0} W^t$.
Then $V=V_0\oplus V_0^b$ and, in the action of $S$ on $V$, the elements of $N$ stabilize $V_0$ and $V_0^b$, while the elements of $S\setminus N$ interchange $V_0$ and $V_0^b$.
 
 \begin{lemma}\label{notClaimOmega}
	$ C_S(a)\cap \left( \bigcap_{t\in T_0} K^t\right)=1$.
\end{lemma}

\begin{proof}
Let $c \in C_S(a)\cap \left( \bigcap_{t\in T_0} K^t\right)$. By rationality, there is some $2$-element $s\in S$ such that $c^s=c$ and $a^s=a^2$. Since no element $s$ in $N$ satisfies the second condition, we derive that $s=gb$ for some $g\in N$. Since $K$ is normal in $H$, $\Core_N(K) = \bigcap_{t\in T_0} K^t$ is normal in $N$. Thus
	$$c=c^s\in  \left(\displaystyle\bigcap_{t\in T_0} K^t \right)^s=\displaystyle\bigcap_{t\in T_0 } K^{tb}.$$
	Then $c\in \bigcap_{t\in T} K^t $, so $c=1$ by  \Cref{Kcorefree}.
\end{proof} 

As a consequence, we obtain the version in this case of \Cref{lem:existd}:
\begin{lemma}\label{lema:existed(ificil)}
   $|T_0|>1$. Equivalently, $C_S(a)\not \subseteq H$.  
\end{lemma}
\begin{proof}
    Suppose that $|T_0|=1$. Then $N=H$ and $[S:H]=2$.
If in addition $K=1$, then $S$ is a rational group of order $48$ which has a normal subgroup isomorphic to $Q_8$ but has not a normal Sylow
$3$-subgroup. A straightforward {\sf GAP} \cite{GAP4} computation shows that such group does not exists. Therefore there is some $k\in K$ with $k\neq 1$. Then $k^b\in H\setminus K$, since otherwise $k\in K\cap K^b=1$, by \Cref{Kcorefree}.   Moreover we may assume that $Kk^b$ has order  $2$, just by substituting $k$ by $k^2$ if $Kk^b$ has order $4$. Thus $Kk^b=Kz$.
This implies that $[k,a]$ acts trivially on $W$ and $W^b$, and hence on $V$, so by faithfulness, $k$ commutes with $a$. This contradicts \Cref{notClaimOmega}.
\end{proof}

 Next we obtain some extra information on the structure of $H$.  \begin{lemma}\label{lema:Q8diagonal}
	There is an element $\ii\in A_2$ such that: 
	\begin{enumerate} 
		\item \label{lema:Q8diagonalc} $a\ii$ has order $3$ and $\ii^{a^2}\ii^a\ii=1$.		
		\item\label{lema:Q8diagonald} $\ii^t=\ii$ for each $t\in T_0$ and $\ii^{b}=\ii^a\ii=\ii^{b^{-1}}$. 
		\item\label{lema:Q8diagonalf} $\ii^2=z$. 
		\item\label{lema:Q8diagonala} $\GEN{\ii,\ii^a}\cong Q_8$ and  $\GEN{\ii,a} \cong \SL(2,3)$.
		\item \label{lema:Q8diagonalb}$H_2=K\rtimes \GEN{\ii,\ii^a}$ and $H=K\rtimes \GEN{\ii,a}$.
		\item\label{lema:Q8diagonale} $H^t= K^t \rtimes \GEN{\ii,\ii^a}$ for each $t\in T$.

	\end{enumerate}
\end{lemma}

\begin{proof}
Consider $N_{S_2}(\GEN{a})$ acting by conjugation on the set $\mathcal{S}_3(A)$ of Sylow $3$-subgroups of $A$.
By Sylow theorem, the cardinality of $\mathcal S_3(A)$ is a power of $2$, which is different from $1$, because $\GEN{Ka}$ is not normal in $H/K$ and hence $\GEN{a}$ is not normal in $S$.
Since the group acting is a $2$-group and $\GEN{a}$ is a fixed point of this action,
there must be another fixed point, say $\GEN{a\ii}$ for some $1\neq \ii\in A_2$.
Then $\ii^{a^2}\ii^a\ii=1$ because $a\ii$ has order $3$ so \eqref{lema:Q8diagonalc} is proved.
 
%Moreover, as $a^2$ is not conjugate to $a$ in $N$, neither $a\ii$ and $(a\ii)^2=a^2\ii^a\ii$ are in $N$. 

Moreover, working modulo $A_2$ it is clear that $(a\ii)^t=a \ii$ for every $t\in T_0$ and  $(a\ii)^b=(a\ii)^2=a^2 \ii^a \ii$.
Thus $\ii^t=\ii$ for each $t\in T_0$ and $\ii^b=\ii^a\ii=(\ii^{-1})^{a^2}$, so $\ii^{b^{-1}}=\ii^b$, since $b^2\in C_{S}(\ii)$.
This  proves  \eqref{lema:Q8diagonald}.  

If $\ii\in K$, then $\ii=\ii^t\in K^t$ and $\ii=(\ii^a\ii)^{tb}\in K^{tb}$ for every $t\in T_0$, so that $1\ne \ii\in \bigcap_{t\in T}K^t =1$, a contradiction. Thus $\ii\not\in K$. Moreover $\ii^2\notin K$, since otherwise $ K\ii$ is an element of order $2$ in $H_2/K$, so $\ii=kz$ for some $k\in K$.
Then $ 1=\ii^{a^2} \ii^a \ii =k^{a^2}k^a kz \in Kz$, a contradiction. 
Therefore $K\ii$ has order $4$. 
Thus $K\ii^2$ has order $2$, and hence $\ii^2=kz$ for some $k\in K$. 
Then $kz=\ii^2= (\ii^2)^t= k^tz$ and $k^{tb}z =(\ii^2)^{tb} =(\ii^{a }\ii)^2 =((\ii^{-1})^{a^2})^2 =(\ii^{-2} )^{a^2} = (k^{-1})^{a^2}z  $ for each $t\in T_0$. This implies that $k  \in \bigcap_{t\in T}K^t$, so $k =1 $ by \Cref{Kcorefree}.  
This proves \eqref{lema:Q8diagonalf}. 

Summarizing, we have that $\ii^4=a^3=1$, $\ii^{a^2}\ii^a\ii=1$ and that $\ii^2=z$ commutes with $a$. 
This implies that $\GEN{a,\ii}\cong \SL(2,3)$ since the above relations form a presentation of $\SL(2,3)$, and the order of $\GEN{\ii,a}$ is at least $24$ because $|\ii|=4$, $|a|=3$ and $\ii^a\not\in \{\ii,\ii^{-1}\}$. This proves \eqref{lema:Q8diagonala} and \eqref{lema:Q8diagonalb} because $\GEN{\ii,\ii^a}$ has order $8$ and does not intersect $K$. Now, \eqref{lema:Q8diagonale} follows from \eqref{lema:Q8diagonalb} and \eqref{lema:Q8diagonald}.

\end{proof}

Fix $\ii$ satisfying the conditions in \Cref{lema:Q8diagonal}.
Observe that $\SL(2,3)$ has six elements of order $4$ and they are not central.
Arguing as in the proof of \Cref{fact:removeCasea}, we derive that there is $u\in W\setminus \{1\}$ such that $u^{\ii}=u^2$.
%	$$W=\GEN{u}\times \GEN{v}, \quad u^i = u^2, \quad v^i=v^{-2}, \quad u^a=u^{-2}v, \quad v^a=u^2v.$$
We denote
	$$\jj=\ii^a, \quad \kk=\jj^a, \quad u_\ii=u, \quad u_{\ii^{-1}}=u_{\ii}^{\jj}, \quad u_\jj=u_\ii^a, \quad u_{\jj^{-1}}=u_{\ii^{-1}}^a, \quad u_\kk=u_{\jj}^a, \qand u_{\kk^{-1}}=u_{  \jj^{-1}}^a.$$
Then  $W=\GEN{u_\ii}\times \GEN{u_{\ii^{-1}}}$,  $I=\{\ii,\ii^{-1},\jj,\jj^{-1},\kk,\kk^{-1}\}$ is the set of elements of order $4$ of $\GEN{a,\ii}$, and for every $w\in W\setminus \{1\}$ and $l\in I$   $$w^l = w^2 \text{ if and only if } w\in \GEN{u_{l}}.$$
Furthermore, the orbits of the action of $H_2$ on $W$ are
	$$\O_\ii = \GEN{u_\ii}\cup \GEN{u_{\ii^{-1}}}\setminus \{1\}, \quad 
	\O_\jj = \GEN{u_\jj}\cup \GEN{u_{\jj^{-1}}}\setminus \{1\} \qand 
	\O_\kk =\GEN{u_\kk}\cup \GEN{u_{\kk^{-1}}}\setminus \{1\},$$
and they are permuted cyclically by the action of $\GEN{a}$, i.e.
	$$\O_\ii^a=\O_\jj, \quad \O_\jj^a=\O_\kk \qand \O_\kk^a=\O_\ii.$$ 
 
 Recall that $V=W^S= \bigoplus_{t\in T}W^t$.  
In the remainder of the proof we fix the following element of $V$:
$$\mathbf{v}=\prod_{t\in T} w_t^t, \text{ with } w_t = \begin{cases} 
u_\ii, & \text{if } t\in T_0; \\
u_{\jj^{-1}}, & \text{if } t=b ; \\
u_{\kk^{-1}} & \text{if } t \in T_0b \setminus{\{b \}}.\end{cases}$$
 By rationality, there is a  $2 $-element $s\in S$ such that $\mathbf{v}^s=\mathbf{v}^2$, which will be fixed throughout.

  We use the bijection $T\to H\backslash S$, $t\mapsto Ht$, and the action of $S$ on right $H$-cosets of $S$ by multiplication on the right, to define an action of $S$ on $T$. 
	 To avoid confusion with multiplication in $S$, we denote the action of $s\in S$ on $t\in T$ by $t^{(s)}$, i.e. $t^{(s)}$ denotes the unique element of $T$ satisfying 
	 	$$Ht^{(s)}=Hts.$$
	 So we have $t^{(1)}=t$ and $t^{(s_1s_2)}=(t^{(s_1)})^{(s_2)}$   for every $s_1,s_2\in S$. 
	 The kernel of this action is the core of $H$ in $S$. As $A$ is normal in $S$ and is contained in $H$, it follows that $t^{(s)}=t$ for every $s\in A$. 
	 As $N$ is normal in $S$, under this action, the elements of $N$ stabilize $T_0$ and $T_0b$, while the elements of $S\setminus N$ interchange $T_0$ and $T_0b$.

\begin{lemma}\label{lemma:(s)action}
$s\in  K \ii T_0$, $s$ stabilizes $T_0$ and $T_0b$, and $b^{(s)}=b$. Moreover,  
$$Kts(t^{(s)})^{-1}=\begin{cases} 
K\ii , & \text{if } t\in T_0; \\
K\jj^{-1} , & \text{if } t= b; \\
K\kk^{-1} , & \text{if } t\in T_0b\setminus \{b\}.\end{cases}$$
Hence, for every $w\in W$,
\begin{equation}\label{accions}
w^{ts}=\begin{cases} 
w^{\ii t^{(s)}}, & \text{if } t\in T_0; \\
w^{\jj^{-1} t^{(s)}}, & \text{if } t= b; \\
w^{\kk^{-1} t^{(s)}}, & \text{if } t\in T_0b\setminus \{b\}.\end{cases}
\end{equation} 
\end{lemma}

\begin{proof}
As $T=T_0\cup T_0b$ is a right transversal of $H$ in $S$ and $H=H_2\rtimes \GEN{a}$, there are unique $h\in H_2$, $i\in \{0,1,2\}$, $d\in T_0$ and $j\in \{0,1\}$ such that $s=ha^id b^j$.

We first prove that $j=0$. 
By means of contradiction suppose that $j=1$. 
By \Cref{lema:existed(ificil)}, $[S:H]>2$; hence there is some $t_0\in T_0\setminus \{1\}$.
Then
	$$b^{(s)}\in T_0, \quad (t_0b)^{(s)}\in T_0, \quad u_{\jj^{-1}}^{bs} = (u_\ii^{b^{(s)}})^2   \qand u_{\kk^{-1}}^{t_0bs}=(u_\ii^{(t_0b)^{(s)}})^2.$$
Therefore $u_{\jj^{-1}}^{bsb^{(s)^{-1}}}=u_\ii^2=u_{\kk^{-1}}^{t_0bs((t_0b)^{(s)})^{-1}}$, and hence the following elements belong to $H\cap N_2=H_2$:
	\begin{eqnarray*}
		h_1&=&bs(b^{(s)})^{-1}a^i=bhdb(b^{(s)})^{-1}, \\ h_2&=&t_0bs((t_0b)^{(s)})^{-1}a^i=t_0bhdb((t_0b)^{(s)})^{-1}.
		\end{eqnarray*}
Thus $u_{\kk^{-1}}=u_{\jj^{-1}}^{h_1h_2^{-1}} \in \O_{\kk}\cap \O_{\jj}=\emptyset$, a contradiction.

Therefore $j=0$, so $s\in N$, and hence $t^{(s)}\in T_0$ for every $t\in T_0$, and $b^{(s)}\in T_0b$.
Thus $u_\ii^{thd(t^{(s)})^{-1}a^{i}}=u_\ii^{ts(t^{(s)})^{-1}}=u_\ii^2\in \O_\ii\cap \O_\ii^{a^i}$. 
Hence $i=0$ and $ts(t^{(s)})^{-1}\in  K\ii$ for each $t\in T_0$. Applying this to $t=1$ yields $s\in  K \ii T_0$. 

If $b^{(s)}\ne b$ then $u_{\jj^{-1}}^{bs(b^{(s)})^{-1}}=u_{\kk^{-1}}^2$. 
Therefore $bs(b^{(s)})^{-1}\in N_2\cap H=H_2$ and $u_{\kk^{-1}}\in \O_{\jj}\cap \O_{\kk}=\emptyset$, a contradiction. 
Thus $b^{(s)}=b$ and $u_{\jj^{-1}}^{bsb^{-1}}=u_{\jj^{-1}}^2$, so that $bsb^{-1}\in  K \jj^{-1}$. 
Moreover, if $t \in  T_0b\setminus \{b\}$ then $u_{\kk^{-1}}^{ts(t^{(s)})^{-1}}=u_{\kk^{-1}}^2$ so that $ts(t^{(s)})^{-1}\in  K\kk^{-1}$. 
 
\end{proof}

For a positive integer $n$ consider  the maps
\begin{eqnarray*}
\alpha_n:S \to  S; & &  x \mapsto  x^{a^{n-1}} x^{a^{n-2}}\dots x^{a^2  }x^a x, \text{ and}\\
\beta_n:S \to  S; &  & x \mapsto  x^{a^{2(n-1)}} x^{a^{2(n-2)}}\dots x^{a^4}x^{a^2} x.
\end{eqnarray*}
Observe that 
	$$(ax)^n=a^n\alpha_n(x), \quad (a^2x)^n=a^{2n}\beta_n(x) \qand \beta_n(x)^b=\alpha_n(x^b).$$  
We now prove by induction on $n$ that for every $w\in W$ we have 
\begin{equation}\label{accionalphan}
w^{t\alpha_n(s)}=\begin{cases} 
w^{\alpha_n(\ii)t^{(\alpha_n(s))}}, & \text{if } t\in T_0; \\
w^{\beta_n(\jj^{-1})t^{(\alpha_n(s))}}, & \text{if } t=b; \\ 
w^{\beta_n(\kk^{-1})t^{(\alpha_n(s))}}, & \text{if } t \in T_0b \setminus{\{b \}}.
\end{cases}
\end{equation}
For $n=1$, simply apply \eqref{accions}. 
For the induction step,  recall that $A$ is contained in the kernel of the action of $S$ on $T$ so that $t^{(a)}=t$ and $t^{(x^a)}=t^{(x)}$ for every $t\in T$ and $x\in S$.
By \Cref{lemma:(s)action}, $s$, and thus also $\alpha_n(s)$, stabilize $T_0$ and $T_0b$ and fix $b$.
If $t\in T_0$ then 
\begin{align*}
w^{t\alpha_{n+1}(s)}& = w^{t\alpha_n(s)^a s}  = (w^{a^{-1}}) ^{t\alpha_n(s) a  s}  
= (w^{a^{-1}}) ^{\alpha_n(\ii)t^{(\alpha_n(s))} as}  = (w^{\alpha_n(\ii)^a})^{t^{(\alpha_n(s))}s} \\ &
=	(w^{\alpha_n(\ii)^a})^{t^{(\alpha_n(s)^a)}s} =(w^{\alpha_n(\ii)^a})^{\ii t^{( \alpha_n(s)^a s)} }  
= ( w^{\alpha_{n+1}(\ii)})^{t^{(\alpha_{n+1}(s))}}
\end{align*}
If $t\in T_0b$, then the induction step is slightly different: taking $l$ either $\jj^{-1}$ or $\kk^{-1}$, depending on whether $t=b$ or not,
\begin{align*}
w^{t\alpha_{n+1}(s)}& = w^{t\alpha_n(s)^a s}  = (w^a)^{t\alpha_n(s) a  s}  
= (w^a)^{\beta_n(l)t^{(\alpha_n(s))} as}  = (w^{\beta_n(l)^{a^{-1}}})^{t^{(\alpha_n(s))}s} \\ &
=	(w^{\beta_n(l)^{a^{-1}}})^{t^{(\alpha_n(s)^a)}s} = 
	(w^{\beta_n(l)^{a^{-1}}})^{lt^{( \alpha_n(s)^a s)} }  
= ( w^{\beta_{n+1}(l)})^{t^{(\alpha_{n+1}(s))}}.
\end{align*}

Let $s_0= bsb^{-1}$. By \Cref{lemma:(s)action}, $s_0\in K\jj^{-1}$. Thus $Ka^2s_0$ has order $3$ in $H/K$ and hence $a^2s_0$ has order  $3\cdot 2^o $ for some non-negative integer $o$. Let $n$ be the smallest even integer greater or equal than $o$. We define $\hat s = \alpha_{2^n}(s)$. Since $n$ is even, $2^n\equiv 1 \mod 3$ and hence $a^2\beta_{2^n}(s_0)=(a^2s_0)^{2^n}$ has order $3$. Thus $\hat s = \alpha_{2^n}(s)=\beta_{2^n}(s_0)^b\in A\cap N_2$ by \Cref{lemma:(s)action}, so it follows that $\hat s\in A_2$. 
This implies that $t^{(\hat s)}=t$ for every $t\in T$. 
%Furthermore 
%	$$t\hat s t^{-1} K = \begin{cases} \ii K, & \text{ if } t\in C; \\ \jj^{-1} K, & \text{if } t=b; \\ \kk^{-1} K, & \text{otherwise}.\end{cases}$$
Furthermore, by \Cref{lema:Q8diagonal}, $\ii^{a^2}\ii^a\ii=(\jj^{-1})^{a}(\jj^{-1})^{a^2}\jj^{-1}=(\kk^{-1})^a(\kk^{-1})^{a^2}\kk^{-1}=1$. 
As $2^{n}\equiv 1 \mod 3$, we  have 
$$\alpha_{2^n}(\ii)=\ii, \quad \beta_{2^n}(\jj^{-1})=\jj^{-1} \qand \beta_{2^n}(\kk^{-1})=\kk^{-1}.$$
Then, by \eqref{accionalphan}, for every $w\in W$,
	$$w^{t\hat s t^{-1}} = \begin{cases} w^{\alpha_{2^n}(\ii)}=w^{\ii}, & \text{if } t\in T_0; \\
		w^{\beta_{2^n}(\jj^{-1})}=w^{\jj^{-1}}, & \text{if } t=b; \\ 
		w^{\beta_{2^n}(\kk^{-1})}=w^{\kk^{-1}}, & \text{if } t \in T_0b \setminus{\{b\}}. \end{cases}$$
In particular $u_\ii^{\hat s}=u_\ii^\ii=u_\ii^2$, so that $\hat s= k \ii$ for some $k\in K$.
If $t\in T_0$ then 
$$w^{ tk t^{-1}\ii} = w^{t\hat s t^{-1}} =w^{\ii}.$$
Furthermore, if $t\in T_0b$ then $t\ii t^{-1} = \ii^a\ii = (\ii^{a^2})^{-1}=\kk^{-1}$ and hence
$$w^{ tkt^{-1} \kk^{-1}} = w^{t\hat s t^{-1}} = \begin{cases} w^{\jj^{-1}} & \text{if } t=b; \\ w^{\kk^{-1}}, & \text{if } t\in T_0b\setminus \{b\}.\end{cases}$$
Then 
$$k\in \bigcap_{t\in T\setminus \{b\}} K^t \qand w^{\jj^{-1}}=w^{b kb^{-1}\kk^{-1}} =w^{bkb^{-1}\jj^{-1} \ii^{-1}} =w^{bkb^{-1} \ii \jj^{-1}}.$$ 
Thus $bkb^{-1}\ii\in K $, so that $bkb^{-1}\in K\ii^{-1}$ and $bk^2b^{-1} \in   Kz$. Hence,
	$$k^2 \in K^bz \cap \bigcap_{t\in T\setminus \{b\}} K^t.$$ 
In particular, $k^2\in \bigcap_{t\in T_0} K^t\setminus \{1\}$, and as $a$ normalizes $K^b$ and commutes with $z$, $[k^2,a]\in \bigcap_{t\in T} K^t=1$. This yields the final contradiction with \Cref{notClaimOmega}. 

\bibliographystyle{amsalpha}
\bibliography{cut1}

\newcommand{\etalchar}[1]{$^{#1}$}
\providecommand{\bysame}{\leavevmode\hbox to3em{\hrulefill}\thinspace}
\providecommand{\MR}{\relax\ifhmode\unskip\space\fi MR }
\providecommand{\MRhref}[2]{%
  \href{http://www.ams.org/mathscinet-getitem?mr=#1}{#2}
}
\providecommand{\href}[2]{#2}
\begin{thebibliography}{BKMdR23}

\bibitem[BKMdR23]{BKMdR}
A.~B{\"a}chle, A.~Kiefer, S.~Maheshwary, and \'{A}. del R\'{i}o,
  \emph{Gruenberg–kegel graphs: cut groups, rational groups and the prime
  graph question}, Forum Mathematicum \textbf{35} (2023), no.~2, 409--429,
  doi:10.1515/forum-2022-0086.

\bibitem[GAP16]{GAP4}
The GAP~Group, \emph{{GAP -- Groups, Algorithms, and Programming, Version
  4.8.3}}, 2016, \url{http://www.gap-system.org}.

\bibitem[GKKM14]{GKKM14}
A.~L. Gavrilyuk, I.~V. Khramtsov, A.~S. Kondrat'ev, and N.~V. Maslova, \emph{On
  realizability of a graph as the prime graph of a finite group}, Sib.
  \`Elektron. Mat. Izv. \textbf{11} (2014), 246--257.

\bibitem[GKL{\etalchar{+}}15]{GKLNS}
A.~Gruber, T.~M. Keller, M.~L. Lewis, K.~Naughton, and B.~Strasser, \emph{A
  characterization of the prime graphs of solvable groups}, J. Algebra
  \textbf{442} (2015), 397--422.

\bibitem[Gow76]{Gow}
R.~Gow, \emph{Groups whose characters are rational-valued}, J. Algebra
  \textbf{40} (1976), no.~1, 280--299.

\bibitem[Heg05]{Heg05}
P.~Heged\H{u}s, \emph{Structure of solvable rational groups}, Proc. London
  Math. Soc. (3) \textbf{90} (2005), no.~2, 439--471.

\bibitem[Luc02]{LucidoTree}
M.~S. Lucido, \emph{Groups in which the prime graph is a tree}, Boll. Unione
  Mat. Ital. Sez. B Artic. Ric. Mat. (8) \textbf{5} (2002), no.~1, 131--148,
  \url{http://www.bdim.eu/item?id=BUMI_2002_8_5B_1_131_0&fmt=pdf}.

\end{thebibliography}

\end{document}